\documentclass[12pt]{amsart}
\usepackage{amssymb,mathdots,mathtools,pgfplots}
\definecolor{hot}{RGB}{65,105,225} 
\usepackage[colorlinks, allcolors=hot]{hyperref}
\usepackage[margin=1in]{geometry}

\pgfplotsset{compat=1.18}

\usepackage{listings}
\usepackage{xcolor}
\definecolor{maccolor}{rgb}{0.3,0.3,0.8}
\lstdefinelanguage{Macaulay2}
{
basicstyle={\ttfamily},
keywordstyle={\color{maccolor!80!black}},
commentstyle={\color{gray}},
stringstyle={\color{red!40!black}},
rulecolor=\color{maccolor},
basewidth={1.2ex}, 
sensitive=false,
morecomment=[l]{--},
morecomment=[s]{-*}{*-},
morestring=[b]",
escapechar={`},
escapebegin={\rmfamily},
morekeywords={about,abs,AbstractToricVarieties,accumulate,Acknowledgement,acos,acosh,acot,addCancelTask,addDependencyTask,addEndFunction,addHook,AdditionalPaths,addStartFunction,addStartTask,Adjacent,adjoint,AdjointIdeal,AffineVariety,AfterEval,AfterNoPrint,AfterPrint,agm,AInfinity,alarm,AlgebraicSplines,Algorithm,Alignment,all,AllCodimensions,allowableThreads,ambient,analyticSpread,Analyzer,AnalyzeSheafOnP1,ancestor,ancestors,ANCHOR,and,andP,AngleBarList,ann,annihilator,antipode,any,append,applicationDirectory,applicationDirectorySuffix,apply,applyKeys,applyPairs,applyTable,applyValues,apropos,argument,Array,arXiv,Ascending,ascii,asin,asinh,ass,assert,associatedGradedRing,associatedPrimes,AssociativeAlgebras,AssociativeExpression,atan,atan2,atEndOfFile,Authors,autoload,AuxiliaryFiles,backtrace,Bag,Bareiss,baseFilename,BaseFunction,baseName,baseRing,baseRings,BaseRow,BasicList,basis,BasisElementLimit,Bayer,BeforePrint,beginDocumentation,BeginningMacaulay2,Benchmark,benchmark,Bertini,BesselJ,BesselY,betti,BettiCharacters,BettiTally,between,BGG,BIBasis,Binary,BinaryOperation,Binomial,binomial,BinomialEdgeIdeals,Binomials,BKZ,BlockMatrix,BLOCKQUOTE,BODY,Body,BoijSoederberg,BOLD,Book3264Examples,Boolean,BooleanGB,borel,Boxes,BR,break,Browse,Bruns,cache,CacheExampleOutput,CacheFunction,CacheTable,cacheValue,CallLimit,cancelTask,capture,catch,Caveat,CC,CDATA,ceiling,Center,centerString,Certification,ChainComplex,chainComplex,ChainComplexExtras,ChainComplexMap,ChainComplexOperations,ChangeMatrix,char,CharacteristicClasses,characters,charAnalyzer,check,CheckDocumentation,chi,Chordal,class,Classic,clean,clearAll,clearEcho,clearOutput,close,closeIn,closeOut,ClosestFit,CODE,code,codim,CodimensionLimit,coefficient,CoefficientRing,coefficientRing,coefficients,Cofactor,CohenEngine,CohenTopLevel,CoherentSheaf,CohomCalg,cohomology,coimage,CoincidentRootLoci,coker,cokernel,collectGarbage,columnAdd,columnate,columnMult,columnPermute,columnRankProfile,columnSwap,combine,Command,commandInterpreter,commandLine,COMMENT,commonest,commonRing,comodule,CompactMatrix,compactMatrixForm,CompiledFunction,CompiledFunctionBody,CompiledFunctionClosure,Complement,complement,complete,CompleteIntersection,CompleteIntersectionResolutions,Complexes,ComplexField,components,compose,compositions,compress,concatenate,conductor,ConductorElement,cone,Configuration,ConformalBlocks,conjugate,connectionCount,Consequences,Constant,Constants,constParser,content,continue,contract,Contributors,ConvexInterface,conwayPolynomial,ConwayPolynomials,copy,copyDirectory,copyFile,copyright,Core,CorrespondenceScrolls,cos,cosh,cot,CotangentSchubert,cotangentSheaf,coth,cover,coverMap,cpuTime,createTask,Cremona,csc,csch,current,currentColumnNumber,currentDirectory,currentFileDirectory,currentFileName,currentLayout,currentLineNumber,currentPackage,currentString,currentTime,Cyclotomic,Database,Date,DD,dd,deadParser,debug,debugError,DebuggingMode,debuggingMode,debugLevel,DecomposableSparseSystems,Decompose,decompose,deepSplice,Default,default,defaultPrecision,Degree,degree,degreeLength,DegreeLift,DegreeLimit,DegreeMap,DegreeOrder,DegreeRank,Degrees,degrees,degreesMonoid,degreesRing,delete,demark,denominator,Dense,Density,Depth,depth,Descending,Descent,Describe,describe,Description,det,determinant,DeterminantalRepresentations,DGAlgebras,diagonalMatrix,diameter,Dictionary,dictionary,dictionaryPath,diff,DiffAlg,difference,dim,directSum,disassemble,discriminant,dismiss,Dispatch,distinguished,DIV,Divide,divideByVariable,DivideConquer,DividedPowers,Divisor,DL,Dmodules,do,doc,docExample,docTemplate,document,DocumentTag,Down,drop,DT,dual,eagonNorthcott,EagonResolution,echoOff,echoOn,EdgeIdeals,edit,EigenSolver,eigenvalues,eigenvectors,eint,EisenbudHunekeVasconcelos,elapsedTime,elapsedTiming,elements,Eliminate,eliminate,Elimination,EliminationMatrices,EllipticCurves,EllipticIntegrals,else,EM,Email,End,end,endl,endPackage,Engine,engineDebugLevel,EngineRing,EngineTests,entries,EnumerationCurves,environment,Equation,EquivariantGB,erase,erf,erfc,error,errorDepth,euler,EulerConstant,eulers,even,EXAMPLE,ExampleFiles,ExampleItem,examples,ExampleSystems,Exclude,exec,exit,exp,expectedReesIdeal,expm1,exponents,export,exportFrom,exportMutable,Expression,expression,Ext,extend,ExteriorIdeals,ExteriorModules,exteriorPower,Factor,factor,false,Fano,FastMinors,FastNonminimal,FGLM,File,fileDictionaries,fileExecutable,fileExists,fileExitHooks,fileLength,fileMode,FileName,FilePosition,fileReadable,fileTime,fileWritable,fillMatrix,findFiles,findHeft,FindOne,findProgram,findSynonyms,FiniteFittingIdeals,First,first,firstkey,FirstPackage,fittingIdeal,flagLookup,FlatMonoid,flatten,flattenRing,Flexible,flip,floor,flush,fold,FollowLinks,for,forceGB,fork,FormalGroupLaws,Format,format,formation,FourierMotzkin,FourTiTwo,fpLLL,frac,fraction,FractionField,frames,FrobeniusThresholds,from,fromDividedPowers,fromDual,Function,FunctionApplication,FunctionBody,functionBody,FunctionClosure,FunctionFieldDesingularization,fusePairs,futureParser,GaloisField,gb,GBDegrees,gbRemove,gbSnapshot,gbTrace,gcd,gcdCoefficients,gcdLLL,GCstats,genera,GeneralOrderedMonoid,GenerateAssertions,generateAssertions,generator,generators,Generic,GenericInitialIdeal,genericMatrix,genericSkewMatrix,genericSymmetricMatrix,gens,genus,get,getc,getChangeMatrix,getenv,getGlobalSymbol,getNetFile,getNonUnit,getPrimeWithRootOfUnity,getSymbol,getWWW,GF,gfanInterface,Givens,GKMVarieties,GLex,Global,global,globalAssign,globalAssignFunction,GlobalAssignHook,globalAssignment,globalAssignmentHooks,GlobalDictionary,GlobalHookStore,globalReleaseFunction,GlobalReleaseHook,Gorenstein,GradedLieAlgebras,GradedModule,gradedModule,GradedModuleMap,gradedModuleMap,gramm,GraphicalModels,GraphicalModelsMLE,Graphics,graphIdeal,graphRing,Graphs,Grassmannian,GRevLex,GroebnerBasis,groebnerBasis,GroebnerBasisOptions,GroebnerStrata,GroebnerWalk,groupID,GroupLex,GroupRevLex,GTZ,Hadamard,handleInterrupts,HardDegreeLimit,hash,HashTable,hashTable,HEAD,HEADER1,HEADER2,HEADER3,HEADER4,HEADER5,HEADER6,HeaderType,Heading,Headline,Heft,heft,Height,height,help,Hermite,hermite,Hermitian,HH,hh,HigherCIOperators,HighestWeights,Hilbert,hilbertFunction,hilbertPolynomial,hilbertSeries,HodgeIntegrals,hold,Holder,Hom,homeDirectory,HomePage,Homogeneous,Homogeneous2,homogenize,homology,homomorphism,HomotopyLieAlgebra,hooks,horizontalJoin,HorizontalSpace,HR,HREF,HTML,html,httpHeaders,Hybrid,HyperplaneArrangements,Hypertext,hypertext,HypertextContainer,HypertextParagraph,icFracP,icFractions,icMap,icPIdeal,id,Ideal,ideal,idealizer,identity,if,IgnoreExampleErrors,ii,image,imaginaryPart,IMG,ImmutableType,importFrom,in,incomparable,Increment,independentSets,indeterminate,IndeterminateNumber,Index,index,indexComponents,IndexedVariable,IndexedVariableTable,indices,inducedMap,inducesWellDefinedMap,InexactField,InexactFieldFamily,InexactNumber,InfiniteNumber,infinity,info,InfoDirSection,infoHelp,Inhomogeneous,input,Inputs,insert,installAssignmentMethod,installedPackages,installHilbertFunction,installMethod,installMinprimes,installPackage,InstallPrefix,instance,instances,IntegralClosure,integralClosure,integrate,IntermediateMarkUpType,interpreterDepth,intersect,intersectInP,Intersection,intersection,interval,InvariantRing,inverse,InverseMethod,inversePermutation,Inverses,inverseSystem,InverseSystems,Invertible,InvolutiveBases,irreducibleCharacteristicSeries,irreducibleDecomposition,isAffineRing,isANumber,isBorel,isCanceled,isCommutative,isConstant,isDirectory,isDirectSum,isEmpty,isField,isFinite,isFinitePrimeField,isFreeModule,isGlobalSymbol,isHomogeneous,isIdeal,isInfinite,isInjective,isInputFile,isIsomorphism,isLinearType,isListener,isLLL,isMember,isModule,isMonomialIdeal,isNormal,isOpen,isOutputFile,isPolynomialRing,isPrimary,isPrime,isPrimitive,isPseudoprime,isQuotientModule,isQuotientOf,isQuotientRing,isReady,isReal,isReduction,isRegularFile,isRing,isSkewCommutative,isSorted,isSquareFree,isStandardGradedPolynomialRing,isSubmodule,isSubquotient,isSubset,isSupportedInZeroLocus,isSurjective,isTable,isUnit,isWellDefined,isWeylAlgebra,ITALIC,Iterate,Jacobian,jacobian,jacobianDual,Jets,Join,join,Jupyter,K3Carpets,K3Surfaces,Keep,KeepFiles,KeepZeroes,ker,kernel,kernelLLL,kernelOfLocalization,Key,keys,Keyword,Keywords,kill,koszul,Kronecker,KustinMiller,LABEL,last,lastMatch,LATER,LatticePolytopes,Layout,lcm,leadCoefficient,leadComponent,leadMonomial,leadTerm,Left,left,length,LengthLimit,letterParser,Lex,LexIdeals,LI,Licenses,LieTypes,lift,liftable,Limit,limitFiles,limitProcesses,Linear,LinearAlgebra,LinearTruncations,lineNumber,lines,LINK,linkFile,List,list,listForm,listLocalSymbols,listSymbols,listUserSymbols,LITERAL,LLL,LLLBases,lngamma,load,loadDepth,LoadDocumentation,loadedFiles,loadedPackages,loadPackage,Local,local,localDictionaries,LocalDictionary,localize,LocalRings,locate,log,log1p,LongPolynomial,lookup,lookupCount,LowerBound,LUdecomposition,M0nbar,M2CODE,Macaulay2Doc,makeDirectory,MakeDocumentation,makeDocumentTag,MakeHTML,MakeInfo,MakeLinks,makePackageIndex,MakePDF,makeS2,Manipulator,map,MapExpression,MapleInterface,markedGB,Markov,MarkUpType,match,mathML,Matrix,matrix,MatrixExpression,Matroids,max,maxAllowableThreads,maxExponent,MaximalRank,maxPosition,MaxReductionCount,MCMApproximations,member,memoize,memoizeClear,memoizeValues,MENU,merge,mergePairs,META,method,MethodFunction,MethodFunctionBinary,MethodFunctionSingle,MethodFunctionWithOptions,methodOptions,methods,midpoint,min,minExponent,mingens,mingle,minimalBetti,MinimalGenerators,MinimalMatrix,minimalPresentation,minimalPresentationMap,minimalPresentationMapInv,MinimalPrimes,minimalPrimes,minimalReduction,Minimize,minimizeFilename,MinimumVersion,minors,minPosition,minPres,minprimes,Minus,minus,Miura,MixedMultiplicity,mkdir,mod,Module,module,ModuleDeformations,modulo,MonodromySolver,Monoid,monoid,MonoidElement,Monomial,MonomialAlgebras,monomialCurveIdeal,MonomialIdeal,monomialIdeal,MonomialIntegerPrograms,MonomialOrbits,MonomialOrder,Monomials,monomials,MonomialSize,monomialSubideal,moveFile,multidegree,multidoc,multigraded,MultigradedBettiTally,MultiGradedRationalMap,multiplicity,MultiplicitySequence,MultiplierIdeals,MultiplierIdealsDim2,MultiprojectiveVarieties,mutable,MutableHashTable,mutableIdentity,MutableList,MutableMatrix,mutableMatrix,NAGtypes,Name,nanosleep,Nauty,NautyGraphs,NCAlgebra,NCLex,needs,needsPackage,Net,net,NetFile,netList,new,newClass,newCoordinateSystem,NewFromMethod,newline,NewMethod,newNetFile,NewOfFromMethod,NewOfMethod,newPackage,newRing,nextkey,nextPrime,nil,NNParser,NoetherianOperators,NoetherNormalization,NonminimalComplexes,nonspaceAnalyzer,NoPrint,norm,normalCone,Normaliz,NormalToricVarieties,not,Nothing,notify,notImplemented,NTL,null,nullaryMethods,nullhomotopy,nullParser,nullSpace,Number,number,NumberedVerticalList,numcols,numColumns,numerator,numeric,NumericalAlgebraicGeometry,NumericalCertification,NumericalImplicitization,NumericalLinearAlgebra,NumericalSchubertCalculus,numericInterval,NumericSolutions,numgens,numRows,numrows,odd,oeis,of,ofClass,OL,OldPolyhedra,OldToricVectorBundles,on,OneExpression,OnlineLookup,OO,oo,ooo,oooo,openDatabase,openDatabaseOut,openFiles,openIn,openInOut,openListener,OpenMath,openOut,openOutAppend,operatorAttributes,Option,OptionalComponentsPresent,optionalSignParser,Options,options,OptionTable,optP,or,Order,order,OrderedMonoid,orP,OutputDictionary,Outputs,override,pack,Package,package,PackageCitations,PackageDictionary,PackageExports,PackageImports,PackageTemplate,packageTemplate,pad,pager,PairLimit,pairs,PairsRemaining,PARA,Parametrization,parent,Parenthesize,Parser,Parsing,part,Partition,partition,partitions,parts,path,pdim,peek,PencilsOfQuadrics,Permanents,permanents,permutations,pfaffians,PHCpack,PhylogeneticTrees,pi,PieriMaps,pivots,PlaneCurveSingularities,plus,poincare,poincareN,Points,polarize,poly,Polyhedra,Polymake,PolynomialRing,Posets,Position,position,positions,PositivityToricBundles,POSIX,Postfix,Power,power,powermod,PRE,Precision,precision,Prefix,prefixDirectory,prefixPath,preimage,prepend,presentation,pretty,primaryComponent,PrimaryDecomposition,primaryDecomposition,PrimaryTag,PrimitiveElement,Print,print,printerr,printingAccuracy,printingLeadLimit,printingPrecision,printingSeparator,printingTimeLimit,printingTrailLimit,printString,printWidth,processID,Product,product,ProductOrder,profile,profileSummary,Program,programPaths,ProgramRun,Proj,Projective,ProjectiveHilbertPolynomial,projectiveHilbertPolynomial,ProjectiveVariety,promote,protect,Prune,prune,PruneComplex,pruningMap,Pseudocode,pseudocode,pseudoRemainder,Pullback,PushForward,pushForward,Python,QQ,QQParser,QRDecomposition,QthPower,Quasidegrees,QuaternaryQuartics,QuillenSuslin,quit,Quotient,quotient,quotientRemainder,QuotientRing,Radical,radical,RadicalCodim1,radicalContainment,RaiseError,random,RandomCanonicalCurves,RandomComplexes,RandomCurves,RandomCurvesOverVerySmallFiniteFields,RandomGenus14Curves,RandomIdeals,randomKRationalPoint,RandomMonomialIdeals,randomMutableMatrix,RandomObjects,RandomPlaneCurves,RandomPoints,RandomSpaceCurves,Range,rank,RationalMaps,RationalPoints,RationalPoints2,ReactionNetworks,read,readDirectory,readlink,readPackage,RealField,RealFP,realPart,realpath,RealQP,RealQP1,RealRoots,RealRR,RealXD,recursionDepth,recursionLimit,Reduce,reducedRowEchelonForm,reduceHilbert,reductionNumber,ReesAlgebra,reesAlgebra,reesAlgebraIdeal,reesIdeal,References,ReflexivePolytopesDB,regex,regexQuote,registerFinalizer,regSeqInIdeal,Regularity,regularity,relations,RelativeCanonicalResolution,relativizeFilename,Reload,remainder,RemakeAllDocumentation,remove,removeDirectory,removeFile,removeLowestDimension,reorganize,replace,RerunExamples,res,reshape,ResidualIntersections,ResLengthThree,Resolution,resolution,ResolutionsOfStanleyReisnerRings,restart,Result,resultant,Resultants,return,returnCode,Reverse,reverse,RevLex,Right,right,Ring,ring,RingElement,RingFamily,ringFromFractions,RingMap,rootPath,roots,rootURI,rotate,round,rowAdd,RowExpression,rowMult,rowPermute,rowRankProfile,rowSwap,RR,RRi,rsort,run,RunDirectory,RunExamples,RunExternalM2,runHooks,runLengthEncode,runProgram,same,saturate,Saturation,scan,scanKeys,scanLines,scanPairs,scanValues,schedule,schreyerOrder,Schubert,Schubert2,SchurComplexes,SchurFunctors,SchurRings,SCRIPT,scriptCommandLine,ScriptedFunctor,SCSCP,searchPath,sec,sech,SectionRing,SeeAlso,seeParsing,SegreClasses,select,selectInSubring,selectVariables,SelfInitializingType,SemidefiniteProgramming,Seminormalization,separate,SeparateExec,separateRegexp,Sequence,sequence,Serialization,serialNumber,Set,set,setEcho,setGroupID,setIOExclusive,setIOSynchronized,setIOUnSynchronized,setRandomSeed,setup,setupEmacs,sheaf,SheafExpression,sheafExt,sheafHom,SheafOfRings,shield,ShimoyamaYokoyama,short,show,showClassStructure,showHtml,showStructure,showTex,showUserStructure,SimpleDoc,simpleDocFrob,SimplicialComplexes,SimplicialDecomposability,SimplicialPosets,SimplifyFractions,sin,singularLocus,sinh,size,size2,SizeLimit,SkewCommutative,SlackIdeals,sleep,SLnEquivariantMatrices,SLPexpressions,SMALL,smithNormalForm,solve,someTerms,Sort,sort,sortColumns,SortStrategy,source,SourceCode,SourceRing,SPACE,SpaceCurves,SPAN,span,SparseMonomialVectorExpression,SparseResultants,SparseVectorExpression,Spec,SpechtModule,SpecialFanoFourfolds,specialFiber,specialFiberIdeal,SpectralSequences,splice,splitWWW,sqrt,SRdeformations,stack,stacksProject,Standard,standardForm,standardPairs,StartWithOneMinor,stashValue,StatePolytope,StatGraphs,status,stderr,stdio,step,StopBeforeComputation,stopIfError,StopWithMinimalGenerators,Strategy,String,STRONG,StronglyStableIdeals,STYLE,Style,style,SUB,sub,SubalgebraBases,sublists,submatrix,submatrixByDegrees,Subnodes,subquotient,SubringLimit,Subscript,subscript,SUBSECTION,subsets,substitute,substring,subtable,Sugarless,Sum,sum,SumOfTwists,SumsOfSquares,SUP,super,SuperLinearAlgebra,Superscript,superscript,support,SVD,SVDComplexes,switch,SwitchingFields,sylvesterMatrix,Symbol,symbol,SymbolBody,symbolBody,SymbolicPowers,symlinkDirectory,symlinkFile,symmetricAlgebra,symmetricAlgebraIdeal,symmetricKernel,SymmetricPolynomials,symmetricPower,synonym,SYNOPSIS,syz,Syzygies,SyzygyLimit,SyzygyMatrix,SyzygyRows,syzygyScheme,TABLE,Table,table,take,Tally,tally,tan,TangentCone,tangentCone,tangentSheaf,tanh,target,Task,taskResult,TateOnProducts,TD,temporaryFileName,tensor,tensorAssociativity,TensorComplexes,terminalParser,terms,TEST,Test,testExample,testHunekeQuestion,TestIdeals,TestInput,tests,TEX,tex,TeXmacs,texMath,Text,TH,then,Thing,ThinSincereQuivers,ThreadedGB,threadVariable,Threshold,throw,Time,time,times,timing,TITLE,TO,to,TO2,toAbsolutePath,toCC,toDividedPowers,toDual,toExternalString,toField,TOH,toList,toLower,top,top,topCoefficients,Topcom,topComponents,topLevelMode,Tor,TorAlgebra,Toric,ToricInvariants,ToricTopology,ToricVectorBundles,toRR,toRRi,toSequence,toString,TotalPairs,toUpper,TR,trace,transpose,TriangularSets,Tries,Trim,trim,Triplets,Tropical,true,Truncate,truncate,truncateOutput,Truncations,try,TSpreadIdeals,TT,tutorial,Type,TypicalValue,typicalValues,UL,ultimate,unbag,uncurry,Undo,undocumented,uniform,uninstallAllPackages,uninstallPackage,Unique,unique,Units,Unmixed,unsequence,unstack,Up,UpdateOnly,UpperTriangular,URL,urlEncode,Usage,use,UseCachedExampleOutput,UseHilbertFunction,UserMode,userSymbols,UseSyzygies,utf8,utf8check,validate,value,values,Variable,VariableBaseName,Variables,Variety,variety,vars,Vasconcelos,Vector,vector,VectorExpression,VectorFields,VectorGraphics,Verbose,Verbosity,Verify,VersalDeformations,versalEmbedding,Version,version,VerticalList,VerticalSpace,viewHelp,VirtualResolutions,VirtualTally,VisibleList,Visualize,wait,WebApp,wedgeProduct,weightRange,Weights,WeylAlgebra,WeylGroups,when,whichGm,while,width,wikipedia,Wrap,wrap,WrapperType,XML,xor,youngest,zero,ZeroExpression,zeta,ZZ,ZZParser,
makeGWClass,getDiagonalClass,makeDiagonalForm,getSignature,isAnisotropic,isIsotropic,getAnisotropicPart,getSumDecompositionString,getGlobalA1Degree,getLocalA1Degree,isIsomorphicForm,addGW,getSumDecomposition,transferGW,makeGWuClass,getGlobalUnstableA1Degree,getLocalUnstableA1Degree,addGWuDivisorial}
}
\lstalias{Macaulay2output}{Macaulay2}

\newtheorem{theorem}{Theorem}
\newtheorem{proposition}[theorem]{Proposition}

\theoremstyle{definition}
\newtheorem{remark}[theorem]{Remark}

\DeclareMathOperator{\Bez}{B\acute{e}z}
\DeclareMathOperator{\GW}{GW}
\DeclareMathOperator{\Nwt}{Nwt}
\DeclareMathOperator{\Res}{Res}
\DeclareMathOperator{\Tr}{Tr}

\title[Transfers and Unstable $\mathbb{A}^{1}$-Brouwer Degrees in Macaulay2]{Transfers and Unstable Degrees in the $\mathbb{A}^{1}$-Brouwer Degrees Package for Macaulay2}

\author[Atherton]{Stephanie Atherton} 
\address{Stephanie Atherton\\
         Department of Mathematics and Statistics\\
         University of Massa\-chusetts\\
         Lowell, MA 01854\\ 
         USA} 
\email{stephanie\_atherton@uml.edu}
\urladdr{https://bio.site/toyTeX}
\author[Dutta]{Somak Dutta} 
\address{Somak Dutta\\
         Department of Mathematics\\
         Texas A\&M University\\
         College Station, TX 77843\\ 
         USA} 
\email{somakdutta@tamu.edu}
\urladdr{https://somak-dutta.github.io}
\author[Lopez Garcia]{Jordy Lopez Garcia} 
\address{Jordy Lopez Garcia\\ 
         Department of Applied and Computational Mathematics and Statistics\\
         University of Notre Dame\\
         Notre Dame, IN 46656\\ 
         USA} 
\email{jlopezga@nd.edu}
\urladdr{https://jordylopez27.github.io}
\author[Louwsma]{Joel Louwsma} 
\address{Joel Louwsma\\
         Department of Mathematics\\ 
         Niagara University\\ 
         Niagara University, NY 14109\\ 
         USA} 
\email{jlouwsma@niagara.edu}
\urladdr{https://www.joellouwsma.com}
\author[Luo]{Yuyuan Luo} 
\address{Yuyuan Luo\\
         Department of Mathematics\\
         Princeton University\\
         Princeton, NJ 08544\\ 
         USA} 
\email{luo.yuyuan@princeton.edu}
\urladdr{https://yuyuan-luo.github.io}
\author[Ong]{Wern Juin Gabriel Ong} 
\address{Wern Juin Gabriel Ong\\
         Department of Mathematics\\
         University of Bonn\\
         Bonn, D-53111\\
         Germany} 
\email{wgabrielong@uni-bonn.de}
\urladdr{https://wgabrielong.github.io}
\author[Sagarayaj]{Ruzho Sagarayaj} 
\address{Ruzho Sagarayaj\\
         Department of Mathematics\\
         Texas A\&M University\\
         College Station, TX 77843\\ 
         USA} 
\email{ruzhomath@tamu.edu}


\begin{document}

\begin{abstract}
We describe a significant update to the \emph{Macaulay2} package \texttt{A1BrouwerDegrees}. We extend several methods in the previous version of the package to the setting of finite \'{e}tale algebras, allowing the computation of transfers along finite \'{e}tale extensions. Additionally, we implement a number of new features for the computation of unstable $\mathbb{A}^{1}$-Brouwer degrees and manipulation of classes in the unstable Grothendieck--Witt group. 
\end{abstract}

\maketitle

\section{Introduction}\label{sec: introduction}

The $\mathbb{A}^{1}$-Brouwer degree is a quadratic-form-valued enhancement of the classical Brouwer degree in motivic homotopy theory. Instead of being valued in the integers, the $\mathbb{A}^{1}$-Brouwer degree (or $\mathbb{A}^{1}$-degree) of a morphism between affine spaces $f\colon\mathbb{A}^{n}_{k}\to\mathbb{A}^{n}_{k}$ over a field~$k$ for is valued in the Grothendieck--Witt ring $\GW(k)$ of symmetric bilinear forms. Algorithms for computing these degrees and manipulating quadratic forms are provided in \emph{Macaulay2}~\cite{M2} as part of the package \texttt{A1BrouwerDegrees} \cite{A1BrouwerDegreesSource,A1BrouwerDegreesArticle}. 

We describe an update to this package that improves on its features in several ways: 
\begin{enumerate}
    \item Computations in the Grothendieck--Witt ring can now be performed over finite \'{e}tale algebras over fields, extending the previous version of the package which was restricted to $\mathbb{C},\mathbb{R},\mathbb{Q}$, and finite fields of characteristic not equal to two. 
    \item For a finite \'{e}tale algebra~$L$ over a field~$k$, we implement a function for computing the transfer $\GW(L)\to\GW(k)$. 
    \item We implement methods for computing unstable $\mathbb{A}^{1}$-Brouwer degrees both globally and locally at $k$-rational points, valued in the unstable Grothendieck--Witt group $\GW^{u}(k)$. 
    \item We implement methods for manipulating unstable Grothendieck--Witt classes over finite \'{e}tale algebras, including methods for verifying isomorphisms of classes and producing simplified representatives.  
\end{enumerate}

\subsection{Outline}\label{subsec: outline}
In Section~\ref{sec: etale and transfers}, we discuss methods for computing with symmetric bilinear forms over a finite \'{e}tale algebra~$L$ over a field~$k$ and for computing the transfer $\GW(L)\to\GW(k)$. Section~\ref{sec: unstable group} discusses the construction of unstable Grothendieck--Witt classes in \emph{Macaulay2} and methods for manipulating them. Finally, Section~\ref{sec: unstable degrees} discusses methods for computing unstable $\mathbb{A}^{1}$-Brouwer degrees and the relationship between local and global degrees via a Poincar\'{e}--Hopf-type theorem.

\subsection{Related work}\label{subsec: related work}
The \emph{Macaulay2} package \texttt{A1BrouwerDegrees} \cite{A1BrouwerDegreesSource,A1BrouwerDegreesArticle} provides methods for computing local and global $\mathbb{A}^{1}$-Brouwer degrees for endomorphisms of affine spaces over fields and for manipulation of quadratic forms in the Grothendieck--Witt ring $\GW(k)$ of a field~$k$, building on preliminary algorithms of Brazelton--McKean--Pauli \cite{BMPCode,BMP}. 

Explicit formulae to compute unstable global $\mathbb{A}^{1}$-degrees were introduced by Cazanave~\cite{Cazanave} and further studied by Kass--Wickelgren~\cite{KWClassical}, who compute the $\GW(k)$-factor of the unstable local degree at rational points. Recent work of Igieobo--McKean--Sanchez--Taylor--Wickelgren~\cite{IMSTW} provides methods for computing local unstable $\mathbb{A}^{1}$-degrees at rational points and provides a Poincar\'{e}--Hopf-type theorem relating local and global unstable $\mathbb{A}^{1}$-degrees for rational functions with all zeroes $k$-rational. It is expected that a Poincar\'{e}--Hopf-type theorem will hold and be similarly computable in the setting where the zeroes of the rational function are not necessarily $k$-rational, as conjectured by Adhikari--Hall--McKean \cite[Conjecture~4.1]{AHM}. 

\subsection{Software availability}\label{subsec: software availibility}
The software described here is available as version 2.0 of the \texttt{A1BrouwerDegrees} package in versions 1.25.11 and later of \emph{Macaulay2}. The documentation is hosted on the \emph{Macaulay2} website at the following link:
{\fontsize{10}{\baselineskip}\selectfont
\begin{center}
\url{https://macaulay2.com/doc/Macaulay2/share/doc/Macaulay2/A1BrouwerDegrees/html/index.html}
\end{center}
\normalsize}

\section{Finite \'{e}tale algebras and Grothendieck--Witt transfers}\label{sec: etale and transfers}

Let $k$ be a field of characteristic not equal to two. Recall that a $k$-algebra~$L$ is said to be of \emph{finite type} if it admits a surjection from a polynomial algebra over~$k$ in finitely many variables \cite[\href{https://stacks.math.columbia.edu/tag/00F3}{Tag~00F3}]{stacks-project}. A finite-type $k$-algebra~$L$ is said to be \emph{finite \'{e}tale} if the trace pairing $L\times L\to k$ given by $(x,y)\mapsto\Tr_{L/k}(xy)$ is non-degenerate \cite[\href{https://stacks.math.columbia.edu/tag/0BVH}{Tag~0BVH}]{stacks-project}. Such an algebra is isomorphic to a finite product of finite separable extensions of~$k$ \cite[\href{https://stacks.math.columbia.edu/tag/00U3}{Tag~00U3}]{stacks-project}.

\subsection{Computations with finite \'{e}tale algebras}\label{subsec: fin et computations}
Given a finite \'{e}tale algebra~$L$ over a field~$k$, we can define its Grothendieck--Witt ring $\GW(L)$. In the \texttt{A1BrouwerDegreesPackage}, elements of Grothendieck--Witt rings have type \texttt{GrothendieckWittClass}. Objects of this type can be constructed from their Gram matrices, which are non-degenerate symmetric matrices over~$L$. 
\begin{lstlisting}[language=Macaulay2output]
`\underline{\tt i1}` : needsPackage "A1BrouwerDegrees"
`\underline{\tt o1}` = `$\texttt{A1BrouwerDegrees}$`
`\underline{\tt o1}` : `$\texttt{Package}$`
`\underline{\tt i2}` : L = QQ[x]/(x^2 - 1);
`\underline{\tt i3}` : M = matrix(L, {{1,2},{2,x}});
`\underline{\tt o3}` : `$\texttt{Matrix}$ $L^{2}\,\longleftarrow \,L^{2}$`
`\underline{\tt i4}` : beta = makeGWClass M
`\underline{\tt o4}` = `$\left(\!\begin{array}{cc}
1&2\\
2&x
\end{array}\!\right)$`
`\underline{\tt o4}` : `$\texttt{GrothendieckWittClass}$`
\end{lstlisting}
The underlying Gram matrix and algebra over which a Grothendieck--Witt class is defined can be retrieved by \texttt{getMatrix} and \texttt{getAlgebra}, respectively. If the class is defined over a field, the field can be extracted using \texttt{getBaseField}. The ring operations on $\GW(L)$ are implemented as \texttt{addGW} and \texttt{multiplyGW}. Since every symmetric matrix is congruent to a diagonal one, every Grothendieck--Witt class admits a diagonal representative. Such representatives can be obtained using \texttt{getDiagonalClass}. 
\begin{lstlisting}[language=Macaulay2output]
`\underline{\tt i5}` : getDiagonalClass beta
`\underline{\tt o5}` = `$\left(\!\begin{array}{cc}
1&0\\
0&x-4
\end{array}\!\right)$`
`\underline{\tt o5}` : `$\texttt{GrothendieckWittClass}$`
\end{lstlisting}

\subsection{Computing transfers}\label{subsec: transfers}
Given a symmetric bilinear form $\beta\colon L\times L\to L$, composition with the trace $\Tr_{L/k}\colon L\to k$ yields a symmetric bilinear form on~$k$. The \emph{transfer} $\Tr_{L/k}$ is the Abelian group homomorphism $\GW(L)\to\GW(k)$  defined by $\beta\mapsto\Tr_{L/k}\circ\beta$. 

We implement a method \texttt{transferGW} for computing the transfer of a Grothendieck--Witt class. This relies on our auxiliary methods \texttt{getMultiplicationMatrix} and \texttt{getTrace} for computing traces of elements of $k$-algebras more generally, which were not previously implemented in \emph{Macaulay2}. 
\begin{lstlisting}[language=Macaulay2output]
`\underline{\tt i6}` : transferGW beta
`\underline{\tt o6}` = `$\left(\!\begin{array}{cc}
2&0\\
0&-8
\end{array}\!\right)$`
`\underline{\tt o6}` : `$\texttt{GrothendieckWittClass}$`
\end{lstlisting}

\section{The unstable Grothendieck--Witt group}\label{sec: unstable group}

In this section, we let $k$ be a field of characteristic not equal to two and $L$ be a finite \'{e}tale algebra over~$k$. We define the \emph{unstable Grothendieck--Witt group} $\GW^{u}(L)$ as the fibered product $\GW(k)\times_{k^{\times}/(k^{\times})^{2}}L^{\times}$ in the catetory of Abelian groups under the determinant and quotient maps, respectively. The operation on $\GW^{u}(L)$ is defined by addition in the first factor and multiplication in the second. Explicit generators and relations for $\GW^{u}(k)$ are described in \cite[Proposition~2.3]{IMSTW}.

The primary way unstable Grothendieck--Witt classes arise is as pointed homotopy classes of endomorphisms of $\mathbb{P}^{1}_{k}$ where $k$ is a field \cite[Theorem~3.6]{Cazanave}.  We further discuss this in Section~\ref{sec: unstable degrees}.

\subsection{Constructing unstable Grothendieck--Witt classes}\label{subsec: constructing GWu classes}
Our newly introduced type \texttt{UnstableGrothendieckWittClass} encodes elements of $\GW^{u}(L)$ as a \texttt{HashTable} consisting of the data of a Gram matrix defined over~$L$ and a scalar in~$L^{\times}$. Objects of this type can be constructed via \texttt{makeGWuClass}, which takes as input either a pair consisting of a symmetric matrix and scalar defined over the same field or a symmetric matrix over a field. In the latter case, the scalar is assumed to be the determinant of the matrix. 
\begin{lstlisting}[language=Macaulay2output]
`\underline{\tt i7}` : M = matrix(L, {{1,2},{2,x}});
`\underline{\tt o7}` : `$\texttt{Matrix}$ $L^{2}\,\longleftarrow \,L^{2}$`
`\underline{\tt i8}` : makeGWuClass M
`\underline{\tt o8}` = `$\left(\left(\!\begin{array}{cc}
1&2\\
2&x
\end{array}\!\right),\,x-4\right)$`
`\underline{\tt o8}` : `$\texttt{UnstableGrothendieckWittClass}$`
`\underline{\tt i9}` : makeGWuClass(M, x - 4)
`\underline{\tt o9}` = `$\left(\left(\!\begin{array}{cc}
1&2\\
2&x
\end{array}\!\right),\,x-4\right)$`
`\underline{\tt o9}` : `$\texttt{UnstableGrothendieckWittClass}$`
\end{lstlisting}
Due to the ubiquity of diagonal and hyperbolic forms, we additionally provide constructors \texttt{makeDiagonalUnstableForm} and \texttt{makeHyperbolicUnstableForm}. 

\subsection{Decompositions and isomorphisms in \texorpdfstring{$\GW^{u}(k)$}{GW\^{}u(k)}}\label{subsec: decompositions and isomorphisms}
In this subsection, we let $k$ be one of $\mathbb{C},\mathbb{R},\mathbb{Q}$, or a finite field of characteristic not equal to two. Two fundamental problems in the study of quadratic forms are to determine when two objects of $\GW(k)$ are isomorphic and to produce simplified representatives of a given quadratic form. Algorithms for decomposing quadratic forms over~$k$ and verifying isomorphisms are discussed in \cite[\S~2.1-2]{A1BrouwerDegreesArticle} and the references therein. These methods readily generalize to the unstable Grothendieck--Witt group $\GW^{u}(k)$.

By the structure of $\GW^{u}(k)$ as a fibered product, two classes in $\GW^{u}(k)$ are isomorphic if and only if their $\GW(k)$-factors are isomorphic and their $k^{\times}$-factors are equal. Testing isomorphisms of forms is implemented as \texttt{isIsomorphicForm}. 
\begin{lstlisting}[language=Macaulay2output]
`\underline{\tt i10}` : A1 = matrix(QQ, {{1,-2,4},{-2,2,0},{4,0,-7}});
`\underline{\tt o10}` : `$\texttt{Matrix}$ ${\mathbb Q}^{3}\,\longleftarrow \,{\mathbb Q}^{3}$`
`\underline{\tt i11}` : A2 = matrix(QQ, {{1,0,0},{0,-2,0},{0,0,9}});
`\underline{\tt o11}` : `$\texttt{Matrix}$ ${\mathbb Q}^{3}\,\longleftarrow \,{\mathbb Q}^{3}$`
`\underline{\tt i12}` : alpha1 = makeGWuClass A1
`\underline{\tt o12}` = `$\left(\left(\!\begin{array}{ccc}
1&-2&4\\
-2&2&0\\
4&0&-7
\end{array}\!\right),\,-18\right)$`
`\underline{\tt o12}` : `$\texttt{UnstableGrothendieckWittClass}$`
`\underline{\tt i13}` : alpha2 = makeGWuClass A2
`\underline{\tt o13}` = `$\left(\left(\!\begin{array}{ccc}
1&0&0\\
0&-2&0\\
0&0&9
\end{array}\!\right),\,-18\right)$`
`\underline{\tt o13}` : `$\texttt{UnstableGrothendieckWittClass}$`
`\underline{\tt i14}` : isIsomorphicForm(alpha1, alpha2)
`\underline{\tt o14}` = `$\texttt{true}$`
\end{lstlisting}
A class in $\GW^{u}(k)$ can be simplified by \texttt{getSumDecomposition}, which finds a diagonal representative of its $\GW(k)$-factor.
\begin{lstlisting}[language=Macaulay2output]
`\underline{\tt i15}` : getSumDecomposition alpha1
`\underline{\tt o15}` = `$\left(\left(\!\begin{array}{ccc}
2&0&0\\
0&1&0\\
0&0&-1
\end{array}\!\right),\,-18\right)$`
`\underline{\tt o15}` : `$\texttt{UnstableGrothendieckWittClass}$`
\end{lstlisting}

\section{Unstable \texorpdfstring{$\mathbb{A}^{1}$}{A\textonesuperior}-Brouwer degrees}\label{sec: unstable degrees}

In this section, we let $k$ be one of $\mathbb{C},\mathbb{Q}$, or a finite field of characteristic not equal to two. A rational function $f/g\colon\mathbb{P}^{1}_{k}\to\mathbb{P}^{1}_{k}$ with monic numerator is said to be \emph{pointed} if $(f/g)(\infty)=\infty$. Given a pointed rational function $f/g\colon\mathbb{P}^{1}_{k}\to\mathbb{P}^{1}_{k}$, work of Cazanave \cite[Theorem~3.6]{Cazanave} and Igieobo--McKean--Sanchez--Taylor--Wickelgren \cite[Lemma~3.10]{IMSTW} shows that the global unstable $\mathbb{A}^{1}$-degree and the local unstable $\mathbb{A}^{1}$-degree at $k$-rational points can be explicitly computed in terms of a bilinear form associated to the rational function. We recall these results below.

\begin{remark}\label{rmk: over CC intro}
In general, our techniques for computing these degrees rely on symbolic computation, which in \emph{Macaulay2} can only be performed over exact fields. Over~$\mathbb{C}$, we are able to find alternative descriptions of the degrees. We describe these in Remarks \ref{rmk: global unstable over CC} and~\ref{rmk: local unstable over CC}. 
\end{remark}

\subsection{Computing global unstable \texorpdfstring{$\mathbb{A}^{1}$}{A\textonesuperior}-degrees}\label{subsec: unstable global degrees} 
Let $f/g\colon\mathbb{P}^{1}_{k}\to\mathbb{P}^{1}_{k}$ be a pointed rational function, where $f$ and~$g$ are polynomials in $k[x]$ with no common factors and $g$ not identically zero. The global unstable $\mathbb{A}^{1}$-degree is given by $(\Bez(f/g),\det\Bez(f/g))\in\GW^{u}(k)$ \cite[Theorem~3.6]{Cazanave}, where $\Bez(f/g)$ is the \emph{B\'{e}zoutian matrix} $\Bez(f/g)=(a_{ij})_{i,j=0}^{m}$ and the $a_{ij}\in k$ are such that 
\[\frac{f(X)g(Y)-f(Y)g(X)}{X-Y}=\sum_{i,j}a_{i,j}X^{i}Y^{j}\in k[X,Y].\]

\begin{remark}\label{rmk: global unstable over CC}
\emph{Macaulay2} does not allow for the construction of rational functions over inexact fields. However, the discussion in \cite[\S~3.2]{Cazanave} shows that determinant of the B\'{e}zoutian matrix can be computed, up to sign, by the resultant of $f$ and~$g$. More specifically, $\det\Bez(f/g)=(-1)^{\frac{n(n-1)}{2}}\cdot\Res(f,g)$, where $n = \deg(f)$. This allows computation of the $k^{\times}$-factor of a global unstable $\mathbb{A}^{1}$-degree over inexact fields. In the case of $\mathbb{C}$, the $\text{GW}(k)$-factor is determined by the degree of a rational function over~$\mathbb{C}$, which together with the $k^{\times}$-factor completely determines a class in $\text{GW}^{u}(\mathbb{C})$. 
\end{remark}

This is implemented as the \texttt{getGlobalUnstableA1Degree} command.
\begin{lstlisting}[language=Macaulay2output]
`\underline{\tt i16}` : frac QQ[x];
`\underline{\tt i17}` : f = x^5 - 6*x^4 + 11*x^3 - 2*x^2 - 12*x + 8;
`\underline{\tt i18}` : g = x^4 - 5*x^2 + 7*x + 1;
`\underline{\tt i19}` : q = f/g;
`\underline{\tt i20}` : GDq = getGlobalUnstableA1Degree q
`\underline{\tt o20}` = `$\left(\left(\!\begin{array}{ccccc}
-68&38&11&-14&1\\
38&-63&63&-29&7\\
11&63&-84&39&-5\\
-14&-29&39&-16&0\\
1&7&-5&0&1
\end{array}\!\right),\,-53\,240\right)$`
`\underline{\tt o20}` : `$\texttt{UnstableGrothendieckWittClass}$`
\end{lstlisting}
An alternate version of this command takes in the numerator and denominator of a rational function separately to accommodate the case of rational functions over~$\mathbb{C}$. When called on rational functions over $\mathbb{Q}$ or a finite field, the computation is run on the underlying quotient. 
\begin{lstlisting}[language=Macaulay2output]
`\underline{\tt i21}` : CC[x];
`\underline{\tt i22}` : fc = x^5 - 6*x^4 + 11*x^3 - 2*x^2 - 12*x + 8;
`\underline{\tt i23}` : gc = x^4 - 5*x^2 + 7*x + 1;
`\underline{\tt i24}` : getGlobalUnstableA1Degree(fc, gc)
`\underline{\tt o24}` = `$\left(\left(\!\begin{array}{ccccc}
1&0&0&0&0\\
0&1&0&0&0\\
0&0&1&0&0\\
0&0&0&1&0\\
0&0&0&0&1
\end{array}\!\right),\,-{53\,240}\right)$`
`\underline{\tt o24}` : `$\texttt{UnstableGrothendieckWittClass}$`
\end{lstlisting}

\subsection{Computing local unstable \texorpdfstring{$\mathbb{A}^{1}$}{A\textonesuperior}-degrees}\label{subsec: unstable local degrees}
For a rational function $f/g\colon\mathbb{P}^{1}_{k}\to\mathbb{P}^{1}_{k}$ with monic numerator that vanishes at $r\in k$ with order~$m$, the local unstable $\mathbb{A}^{1}$-degree is given by the pair $(\Nwt_{r}(f/g),\det\Nwt_{r}(f/g))$, where $\Nwt_{r}(f/g)$ is the \textit{local Newton matrix} at~$r$. This matrix may be expressed as the antidiagonal matrix
\[\Nwt_{r}(f/g)=\begin{bmatrix}
    0 & \dots & 0 & a_{m} \\
    \vdots & \iddots & \iddots & 0 \\
    0 & \iddots & \iddots & \vdots \\
    a_{m} & 0 & \dots & 0
\end{bmatrix},\]
where $a_{m} \in k$ is the coefficient appearing in the Laurent series expansion 
\begin{equation}\label{eqn: Laurent expansion}
    \frac{g(x)}{f(x)}=\frac{a_{m}}{(x-r)^{m}}+\frac{a_{m-1}}{(x-r)^{m-1}}+\dots+\frac{a_{1}}{(x-r)}+\text{higher order terms}
\end{equation}
of $g/f$ at~$r$ \cite[Lemma~3.10]{IMSTW}. 

The method for finding the values $a_{m}\in k$ of the Laurent series expansion is not entirely immediate, and we record the following result -- likely well-known to experts -- for their computation. 

\begin{proposition}\label{prop: computation of higher residues}
Let $f/g$ be a pointed rational function. Let $r\in k$ be a root of~$f$ of multiplicity~$m$. Write a partial fraction decomposition of $g/f$ as in \eqref{eqn: Laurent expansion}. Then $a_{m}$ is given by the evaluation of $\frac{(x-r)^{m}g(x)}{f(x)}$ at~$r$.
\end{proposition}

\begin{proof}
Writing $F(x)=\frac{(x-r)^{m}g(x)}{f(x)}$, we have that $r$ is neither a root nor a pole of $F(x)$ since $f/g$ is pointed. That is, $F$ is an analytic function in the sense that it is an element of the stalk $\mathcal{O}_{\mathbb{P}^{1}_{k},r}$ at~$r$. Since $r$ is a $k$-rational point, it in particular has separable residue field, and we can use the Hasse-derivative analogue of Taylor's theorem \cite[Corollary~2.5.14]{Goldschmidt}, where the power series expansion of~$F$ is given by
\[F(x)=F(r)+\sum_{a\geq 1}D^{(a)}_{x}(F)(r)\cdot (x-r)^{a}.\]
Dividing by $(x-r)^{m}$ once more and comparing coefficients yields the claim. 
\end{proof}

The procedure in Proposition~\ref{prop: computation of higher residues} is implemented as \texttt{getLocalUnstableA1Degree}. Taking \texttt{q} to be as in line \texttt{i19}, we have:
\begin{lstlisting}[language=Macaulay2output]
`\underline{\tt i25}` : deg1 = getLocalUnstableA1Degree(q, -1)
`\underline{\tt o25}` = `$\left(\left(\!\begin{array}{c}
-\frac{5}{27}
\end{array}\!\right),\,-\frac{5}{27}\right)$`
`\underline{\tt o25}` : `$\texttt{UnstableGrothendieckWittClass}$`
\end{lstlisting}

\begin{remark}\label{rmk: local unstable over CC}
Over $\mathbb{C}$, an adapted version of Proposition~\ref{prop: computation of higher residues} can be implemented using numerical algebraic geometry to compute the unstable local $\mathbb{A}^{1}$-degree.
\end{remark}

As in the case of computing global unstable $\mathbb{A}^{1}$-degrees, there is also a variant that takes in the numerator and denominator separately. 
\begin{lstlisting}[language=Macaulay2output]
`\underline{\tt i26}` : getLocalUnstableA1Degree(f, g, -1)
`\underline{\tt o26}` = `$\left(\left(\!\begin{array}{c}
-\frac{5}{27}
\end{array}\!\right),\,-\frac{5}{27}\right)$`
`\underline{\tt o26}` : `$\texttt{UnstableGrothendieckWittClass}$`
\end{lstlisting}

\subsection{A Poincar\'{e}--Hopf theorem at rational points}\label{subec: PH rational points}
The global degree of Section~\ref{subsec: unstable global degrees} and the local degree of Section~\ref{subsec: unstable local degrees} are related by a Poincar\'{e}--Hopf type formula. Unlike in the stable setting, where the global $\mathbb{A}^{1}$-degree is computed as a sum of local $\mathbb{A}^{1}$-degrees in the Grothendieck--Witt ring $\GW(k)$, the local-to-global principle for unstable $\mathbb{A}^{1}$-degrees depends on the configuration of zeroes of $f/g$, that is, on a (Weil) divisor on $\mathbb{P}^{1}_{k}$. More precisely, for a rational function $f/g$ all of whose zeroes $r_{1},\dots,r_{n}$ are $k$-rational with multiplicities~$m_{i}$ and local unstable degrees $(\beta_{i},d_{i})$, the global unstable degree is shown in \cite[Theorem~1.1]{IMSTW} to be given by the divisorial sum formula
\[\deg^{u}(f/g) = \left(\bigoplus_{i=1}^{n}\beta_{i},\prod_{i=1}^{n}d_{i}\cdot\prod_{i<j}(r_{i}-r_{j})^{2m_{i}m_{j}}\right).\]
Continuing with the example introduced in Sections \ref{subsec: unstable global degrees} and~\ref{subsec: unstable local degrees}, we can compute the local unstable $\mathbb{A}^{1}$-degree of~\texttt{q} at its remaining roots $1,2\in\mathbb{Q}$.
\begin{lstlisting}[language=Macaulay2output]
`\underline{\tt i27}` : deg2 = getLocalUnstableA1Degree(q, 1)
`\underline{\tt o27}` = `$\left(\left(\!\begin{array}{c}
-2
\end{array}\!\right),\,-2\right)$`
`\underline{\tt o27}` : `$\texttt{UnstableGrothendieckWittClass}$`
`\underline{\tt i28}` : deg3 = getLocalUnstableA1Degree(q, 2)
`\underline{\tt o28}` = `$\left(\left(\!\begin{array}{ccc}
0&0&\frac{11}{3}\\
0&\frac{11}{3}&0\\
\frac{11}{3}&0&0
\end{array}\!\right),\,-\frac{1\,331}{27}\right)$`
`\underline{\tt o28}` : `$\texttt{UnstableGrothendieckWittClass}$`
\end{lstlisting}
Computing the divisorial sum of the local degrees \texttt{deg1}, \texttt{deg2}, and \texttt{deg3}, this agrees with the global unstable $\mathbb{A}^{1}$-degree computed in Section~\ref{subsec: unstable global degrees}. 
\begin{lstlisting}[language=Macaulay2output]
`\underline{\tt i29}` : degSum = addGWuDivisorial({deg1,deg2,deg3}, {-1,1,2});
`\underline{\tt i30}` : isIsomorphicForm(degSum, GDq)
`\underline{\tt o30}` = `$\texttt{true}$`
\end{lstlisting}

\section{Acknowledgments}\label{sec: acknowledgments}

We warmly thank the organizers of the Madison 2025 \emph{Macaulay2} workshop, supported by NSF DMS-2508868, for providing the opportunity to collaborate on this package update. Sabrina Pauli's vision laid the foundation for this update, and we are especially indebted to her both for this and for several insightful discussions. The third-named author additionally thanks Jonathan Hauenstein for valuable conversations, and the sixth-named author additionally thanks Thomas Brazelton for valuable conversations. 

\bibliographystyle{abbrv}
\bibliography{refs}

\end{document}